\documentclass[10pt,psamsfonts]{amsart}

\usepackage{amsmath,amsthm}     
\usepackage{graphicx}     
\usepackage{hyperref} 
\usepackage{url}
\usepackage{amsfonts} 
\usepackage{url}
\usepackage{chngcntr}
 \usepackage{float}
\usepackage{multicol}
\usepackage{enumerate}
\usepackage{mathtools}
\usepackage{pgf,tikz,pgfplots}
\usepackage{xcolor}
\usetikzlibrary{calc}

\usepackage{soul}

\theoremstyle{theorem}
\newtheorem{theorem}{Theorem}

\newtheorem{lemma}{Lemma}
\newtheorem{proposition}{Proposition}
\newtheorem{corollary}{Corollary}

\theoremstyle{definition}
\newtheorem*{definition}{Definition}

\theoremstyle{remark}

\theoremstyle{example}
\newtheorem{example}{Example}

\allowdisplaybreaks

\makeatletter
\@addtoreset{footnote}{page}
\makeatother

\author{Christian Aebi and Grant Cairns}

\address{Coll\`ege Calvin, Geneva, Switzerland 1211}
\email{christian.aebi@edu.ge.ch}
\address{Department of Mathematical and Physical Sciences, La Trobe University, Melbourne, Australia 3086}
\email{G.Cairns@latrobe.edu.au}

\def\N{\mathbb N}
\def\Z{\mathbb Z}
\def\R{\mathbb R}

\title{Lattice triangles whose centers are lattice points}

\begin{document}

\maketitle

\begin{abstract}
We show that for an integer $\ell$, there exists an acute integer lattice triangle of lattice perimeter $\ell$ such that its orthocenter is an integer lattice point, if and only if $\ell=6 $ or $\ell\ge 8$.
Analogous results are obtained for the circumcenter and the centroid, and the results are  contrasted with those for obtuse and right triangles. 
\end{abstract}

\vspace{10px}


\vspace{10px}

\section{Introduction.}\label{S:intro}

For convenience, the elements of the integer lattice $\Z^2\subseteq \R^2$ will simply be called \emph{lattice points}. The \emph{lattice length} of a line segment $\gamma$ joining two lattice points is defined to be the number of lattice points on $\gamma$ minus one \cite{Rab,HSk,WeiDing}. 
A \emph{lattice polygon} is a planar polygon whose vertices are lattice points. The \emph{lattice perimeter} of a lattice polygon is the sum of the lattice lengths of its sides \cite{Worm,HaaseNil}; so the lattice perimeter is just the number of lattice points on the boundary. Lattice perimeter plays a key role in Pick's famous theorem, which gives the area of a lattice polygon in terms of its lattice perimeter and the number, called the genus, of lattice points in the interior of the polygon \cite{Scott}.

This paper concerns the three classical centers of a triangle: the circumcenter, the centroid, and the orthocenter, which we denote $F,G,H$, respectively. We are interested in knowing what lattice perimeters are possible when one of these centers is itself a lattice point. 

What is the motivation for this study? As authors of an earlier Monthly article wrote, ``Once you get started drawing lattice polygons
on graph paper and discovering relations between their numerical invariants, it is not
so easy to stop! (The gentle reader has been warned.)'' \cite{HaaSc}. The Monthly has a history of publishing papers on lattice polygons (see \cite{Abrams}), and particularly on the concept of lattice perimeter (see \cite{PoRV}). The particular problem investigated in this paper involves the intersection of two distinct domains. The notion of lattice perimeter lies within the field of lattice or integer geometry. The notions of circumcenter and orthocenter lie in the field of Euclidean geometry. These two fields are quite different, especially as concerns their underlying transformations groups; apart from translations, which are common to both of them, lattice geometry uses the general linear group $GL(2,\Z)$, while Euclidean geometry employs the orthogonal group $O(2,\R)$. Questions in the intersection of these two fields can reveal interesting -- and sometimes amusing -- number theoretic problems  while still retaining an elementary character. For instance, in Theorem~\ref{T:3457}, the acute lattice triangles that don't satisfy the stated condition necessarily have lattice perimeter $3,4,5$ or $7$.  It was the appearance of these curious numbers that first caught the authors' attention.

Let us first consider the orthocenter $H$. For general triangles, the hypothesis that $H$ is a lattice point has no consequence for the lattice perimeter $\ell$, as we will see in Section~\ref{S:obright}. The entertaining case is when the triangle is acute. 

\begin{theorem}\label{T:3457}
If $\ell\in\N$, then there exists an acute lattice triangle of lattice perimeter $\ell$, for which the orthocenter $H$ is a lattice point, if and only if 
$\ell=6 $ or $\ell\ge 8$.
\end{theorem}

What is so special about the unusual condition ``$\ell=6 $ or $\ell\ge 8$'' in the above theorem? It does appear in the following curiosity, which is presumably ``well known''.

\begin{proposition}\label{P:3457}
If $n\in\N$, then there exist distinct, pairwise coprime, positive integers $x,y,z$ such that $n=x+y+z$ if and only if 
 $n=6 $ or $n\ge 8$.
\end{proposition}

We actually won't use this proposition in the proof we give of Theorem~\ref{T:3457} in Section~\ref{S:3457}. Nevertheless, we will see that this proposition is closely related to one of the key ingredients in our proof. We leave the proof of Proposition~\ref{P:3457} to the interested reader (Hint: it suffices to consider $x=1$).

Analogous to Theorem~\ref{T:3457}, the results for the circumcenter and the centroid are similar, with the condition ``$\ell=6 $ or $\ell\ge 8$'' replaced by other restrictions. These results are given in the following two theorems.

\begin{theorem}\label{T:4610}
If $\ell\in\N$, then there exists an acute lattice triangle of lattice perimeter $\ell$, for which the circumcenter $F$ is a lattice point, if and only if 
$\ell$ is even and either $\ell=8 $ or $\ell\ge 12$.
\end{theorem}

\begin{theorem}\label{T:511}
If $\ell\in\N$, then there exists an acute  lattice triangle of lattice perimeter $\ell$, for which the centroid $G$ is a lattice point, if and only if 
$\ell\ge 3$ and $\ell\not\in\{5,11\}$.
\end{theorem}

The condition ``$\ell=8 $ or $\ell\ge 12$'' of Theorem~\ref{T:4610} results from the fact that the angles of the triangle can be written as the $\arctan$ of certain rational numbers; see Lemma~\ref{L:theta}. 
The condition ``$\ell\ge 3$ and $\ell\not\in\{5,11\}$'' of Theorem~\ref{T:511} is closely related to the following, second curiosity.

\begin{proposition}\label{P:511}
If $n\in\N$, then there exist pairwise coprime, positive integers $x,y,z$, none of which is divisible by $3$, such that $n=x+y+z$, if and only if 
$n\ge 3$ and  $n\not\in\{5,11\}$.
\end{proposition}

Note that in this proposition, unlike Proposition~\ref{P:3457}, we do not assume that $x,y,z$ are distinct (so we can have $x=y=1$, for example) but we do suppose that they are coprime with $3$. Like Proposition~\ref{P:3457}, we do not actually use Proposition~\ref{P:511} in the paper, and we leave its proof to the interested reader (Hint: same hint as for Proposition~\ref{P:3457}).

The bulk of this paper is devoted to the proofs of Theorems~\ref{T:3457}, \ref{T:4610} and \ref{T:511}. In Section~\ref{S:prelim}, we give some preliminary results. Theorems~\ref{T:3457}, \ref{T:4610} and \ref{T:511} are proved in Sections~\ref{S:3457}, \ref{S:4610} and \ref{S:511}, respectively. 
In Section~\ref{S:obright}, we contrast the above results with those for obtuse and right triangles. We also discuss the consequences of assuming that more than one of the triangle's three centers is on the lattice. Finally, in Section~\ref{S:open}, we pose an open problem concerning the incenter of triangles.


\section{Preliminary observations.}\label{S:prelim}

Notice that if $A=(x,y)\in\Z^2$, then the lattice length of the segment $OA$ is $\gcd(x,y)$.

\begin{lemma}\label{L:gcd}
Suppose that $T$ is a lattice triangle with sides of lattice length $\ell_0,\ell_1,\ell_2$. 
Then $\gcd(\ell_i,\ell_j)=\gcd(\ell_0,\ell_1,\ell_2)$, for all $i\not= j$.
\end{lemma}

\begin{proof}
Suppose that $T$ has vertices $V_i=(x_i,y_i)$, for $i=0,1,2$, such that the side $V_{i+1}V_{i+2}$ has lattice length $\ell_i$, where the subscripts are computed modulo $3$. Note that $\ell_i=\gcd(x_{i+1}-x_{i-1},y_{i+1}-y_{i-1})$, for $i=0,1,2$. With the obvious notation, we may write $\ell_i=\gcd(V_{i+1}-V_{i-1})$. Now consider $\ell_1,\ell_2$. Let $\gcd(\ell_1,\ell_2)=n$. Obviously, $\gcd(\ell_0,\ell_1,\ell_2)$ is a divisor of $n$. We have that $n$ divides both $V_{2}-V_0$ and $V_{0}-V_1$. Hence $n$ divides $V_{2}-V_1$, and thus $n$ divides $\ell_0$. Hence $n$ divides $\gcd(\ell_0,\ell_1,\ell_2)$ and thus $\gcd(\ell_1,\ell_2)=\gcd(\ell_0,\ell_1,\ell_2)$. By the same reasoning, $\gcd(\ell_i,\ell_j)=\gcd(\ell_0,\ell_1,\ell_2)$, for all $i\not= j$.
 \end{proof}

Let us now recall the formulas for the  circumcenter $F$, the centroid $G$, and the orthocenter $H$. Consider the triangle $T$ with vertices $0,A=(x_1,y_1),B=(x_2,y_2)$. One has 
\begin{align}
F&=\frac12(3G-H),\label{E:F}\\
G&=\frac13(x_1+x_2,y_1+y_2),\label{E:G}\\
H&=\frac{x_1x_2+y_1y_2}{x_1y_2-x_2y_1}( y_2-y_1, x_1-x_2).\label{E:H}
\end{align}
The formula for $F$
 is a rearranged form of the classic Euler line equation, $2F+H=3G$, \cite[p.~71]{AN}. The formula for $G$ is just its definition. The formula for $H$ can be verified by showing that the dot products $H\cdot (A - B),(H - A)\cdot B,(H - B)\cdot A$ 
 are all zero. Figure~\ref{F:euler} gives an example of a lattice triangle $OAB$  for which $F,G,H$ are lattice
points. It has $A=(3,9),B=(0,6)$ and  $F=(4,3),G=(3,3),H=(1,3)$.

\begin{figure}
\begin{center}
\begin{tikzpicture}[x=.7cm,y=.7cm]

\draw [fill, blue!20] (0,0) -- (0,6)-- (9,3) -- cycle ;
\draw [blue,semithick] (0,0) -- (0,6)-- (9,3) -- cycle ;

\draw[fill, black] (0,0) circle (.04);
\draw[fill, black] (0,6) circle (.04);
\draw[fill, black] (9,3) circle (.04);
\draw[fill, black] (3,3) circle (.04);
\draw[fill, black] (4,3) circle (.04);
\draw[fill, black] (1,3) circle (.04);

\draw [black,font=\small]  (-.2,-.2) node {$O$} ;
\draw [black,font=\small]  (-.3,6.3) node {$B$} ;
\draw [black,font=\small]  (9.3,3.3) node {$A$} ;
\draw [black,font=\small]  (1.3,3.3) node {$H$} ;
\draw [black,font=\small]  (3.3,3.3) node {$G$} ;
\draw [black,font=\small]  (4.3,3.3) node {$F$} ;

\draw[-,blue,thick] (-1,3) -- (10,3);
\draw[-,blue,dotted,thick] (0,0) -- (1.8,3*1.8);
\draw[-,blue,dotted,thick] (0,6) -- (1.8,6-3*1.8);
\draw[-] (1.7,3*1.7)--(1.7+.3,3*1.7-.1);
\draw[-] (1.8+.3,3*1.8-.1)--(1.7+.3,3*1.7-.1);
\draw[-] (1.7,6-3*1.7)--(1.7+.3,6.1-3*1.7);
\draw[-] (1.8+.3,6.1-3*1.8)--(1.7+.3,6.1-3*1.7);

\draw[black,semithick] (4,3) circle (5);

\foreach \ii in {-1,...,10}
\draw[-,dotted,semithick] (\ii,-1) -- (\ii,7);
\foreach \jj in {-1,...,7}
\draw[-,dotted,semithick] (-1,\jj) -- (10,\jj);

\draw[->,semithick] (-1,0) -- (10,0);
\draw[->,semithick] (0,-1) -- (0,7);

  \end{tikzpicture}
\caption{A lattice triangle for which the circumcenter $F$, centroid $G$, and orthocenter $H$ are lattice
points.}\label{F:euler}
\end{center}
\end{figure}
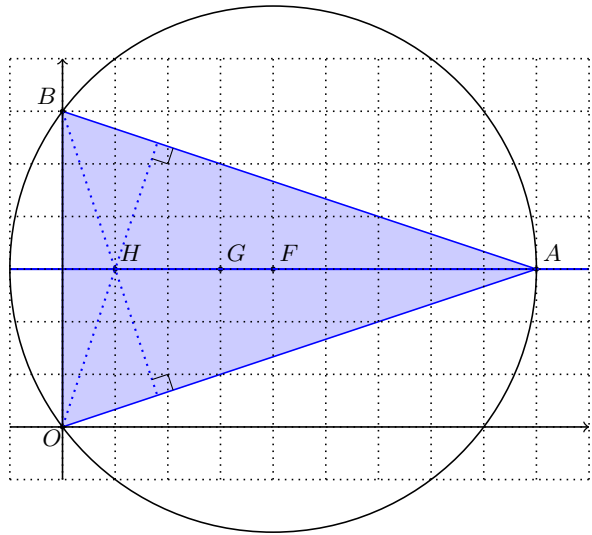

\begin{lemma}\label{L:theta}
Suppose that $T$ is a lattice triangle for which the orthocenter $H$ is a lattice point. Let $\theta$ be an interior angle of $T$ at some vertex $V$, and let $\ell_\theta$ denote the lattice length of the side opposite $\theta$. Then $\theta=\arctan\left(\frac{\ell_\theta}m\right)$,  where $m\in\N$ is the lattice length of the segment $VH$.
\end{lemma}

\begin{proof} Without loss of generality, we may take $V$ to be the origin.
Suppose that $T$ has vertices $0,A=(x_1,y_1),B=(x_2,y_2)$ and that $OAB$ is positively oriented with angle $\theta$ at $O$. Let $OA$ have length $a$ and $OB$ have length $b$. By \eqref{E:H},
\[
H=\frac{x_1x_2+y_1y_2}{x_1y_2-x_2y_1}( y_2-y_1, x_1-x_2)=\frac{x_1x_2+y_1y_2}{x_1y_2-x_2y_1}R(A-B),
\]
where $R( - )$ denotes the rotation through angle $\frac{\pi}2$. Hence, since $ x_1x_2+y_1y_2=ab\cos \theta$ and $x_1y_2-x_2y_1=ab\sin \theta$,
we have
\[
H=\cot \theta \ R(A-B).
\]
In particular, the length  $|OH|$ of the segment $OH$ is given by $|OH|=\cot \theta\  |AB|$. Now $AB$ has lattice length $\ell_\theta$, and  rotation through $\frac{\pi}2$ respects lattice points. It follows that $OH$ has lattice length $\ell_\theta \cot \theta$. Denoting this integer $m$, we have $\tan(\theta)=\frac{\ell_\theta}m$, as required.
\end{proof}

\begin{lemma}\label{L:11n}
Suppose that $T$ is an acute lattice triangle for which the sides have lattice length $1,1,m$ respectively, for some $m\in\N$. Then the orthocenter $H$ is not a lattice point. 
\end{lemma}

\begin{proof}
Arguing by contradiction, suppose that $H$ is a lattice point. Let the angles opposite the sides of  lattice length $1,1,m$ be $\theta_0,\theta_1,\theta_2$, respectively. For $i=0,1$, by the previous lemma, $\tan(\theta_i)=\frac{1}{m_i}$, for some $m_i\in\N$. In particular, $\tan(\theta_i)\le 1$ and so $\theta_i\le \frac{\pi}4$, for $i=0,1$. Thus, $\theta_2=\pi-\theta_0-\theta_1\ge \frac{\pi}2$, contradicting the assumption that $T$ is acute.
\end{proof}

Recall that geometrically one  obtains the circumcenter $F$ of a triangle $OAB$ by considering the intersection of two of its perpendicular bisectors.  The following notion will be convenient.

\begin{definition} We will say that a lattice point $X\in\Z^2$ is \emph{even}, \emph{odd} or \emph{mixed} if respectively, the coordinates of $X$ are both even, both odd, or $X$ has one odd and even coordinate.
\end{definition}

\begin{lemma}\label{L:perp}
If $A$ is a mixed lattice point,  then the perpendicular bisector of $OA$ never intersects $\Z^2$.
\end{lemma}
\begin{proof}
Without loss of generality, suppose that $A=(2m,2n+1)$ with $m, n \in \Z$. The bisector in question has equation  $2(2n+1)y+4mx=(2n+1)^2+4m^2$. The LHS being even and the RHS being odd,
 there is no integer solution to this  equation.
\end{proof}

\begin{corollary}\label{C:cor}
Suppose a lattice triangle $T$ has vertices $O,A,B$, and that its circumcenter $F$ is a lattice point.  Then $A,B$ are either even or odd. Moreover the lattice perimeter $\ell$ of $T$ is even, and its area $K$ is an integer.
\end{corollary}

\begin{proof}
The first assertion is a direct consequence of Lemma~\ref{L:perp}. 
Notice that the side $OA$ has even lattice length if and only if $A$ is even, and a similar statement holds for $B$. Furthermore, 
the side $AB$ has even lattice length if  and only if the vertices $A,B$ have the same parity; i.e, they are either both even or both odd. It follows that either all the sides have even lattice length, or exactly one side has even lattice length. Hence, $\ell$ is even. By Pick's Theorem \cite[pp.39--41]{AN2}, $K= \frac{\ell}2+g-1$, where $g$ is the genus of $T$. Hence, $K$ is an integer. \end{proof}


\begin{lemma}\label{L:angle} Suppose that $T$ is a lattice triangle with vertices $O,A(x_1,y_1),B(x_2,y_2)$, and let $\ell_O$ denote the lattice length of the side $AB$.
If the circumcenter $F$ of $T$ is a lattice point, then 
\begin{enumerate}
\item[(a)] the orthocenter $H$ of $T$ is a lattice point,
\item[(b)] if $\ell_O$ is even, and if $\theta$ denotes the angle of $T$ at $O$, then the lattice length of the segment $OH$ is even,
and $\tan \theta \le \frac{\ell_O}2$.
\end{enumerate}
\end{lemma}
\begin{proof}(a) follows immediately from \eqref{E:F}.

(b). From \eqref{E:G} and the formula $2F+H=3G$, we have
\begin{equation}\label{E:H2}
H=(x_1+x_2,y_1+y_2)-2F.
\end{equation}
By Corollary~\ref{C:cor}, as $\ell_O$ is even, the vertices $A$ and $B$ are either both even or both odd. Hence $x_1+x_2$ and $y_1+y_2$ are both even. Hence, by \eqref{E:H2}, $H\in2\Z^2$. Thus the lattice length of $OH$ is even.
In particular, since the lattice length of $OH$ is even, it is at least $2$, and hence $\tan \theta \le \frac{\ell_O}2$ by Lemma~\ref{L:theta}.
\end{proof}

\begin{lemma}\label{L:mid3} Suppose that $T$ is an acute lattice triangle and that the circumcenter $F$ of $T$ is a lattice point. Let $\ell_0,\ell_1,\ell_2$ be the lattice lengths of the sides of $T$, in increasing order. Then $\ell_1\ge 3$.
\end{lemma}
\begin{proof}
Let the angle opposite the side of lattice length $\ell_i$ be $\theta_i$, for $i=0,1,2$, respectively. 
Suppose $\ell_1\le 2$. 
If $\ell_1=1$, then $\ell_0=1$ and this is impossible by Lemmas~\ref{L:11n} and \ref{L:angle}(a).
Thus, $\ell_1=2$.
Then by Lemma~\ref{L:angle}(b), $\tan \theta_1 \le \frac{\ell_{1}}2\le 1$, and so $\theta_1\le \pi/4$.
Similarly, if $\ell_0=2$, then $\theta_0\le \pi/4$. Moreover, if $\ell_0=1$, then $\tan \theta_0\le \frac{\ell_{1}}1\le 1$, and so $\theta_0\le \pi/4$, by Lemma~\ref{L:theta}.  But then $\theta_2=\pi-\theta_0-\theta_1\ge \frac{\pi}2$, contradicting the assumption that $T$ is acute. 
\end{proof}

\begin{lemma}\label{L:centroid3s} Suppose that $T$ is a lattice triangle with its centroid $G$ on the lattice. If the lattice length of one of the sides of $T$ is a multiple of $3$, then the lattice length of each of the sides of $T$ is a multiple of $3$.
\end{lemma}

\begin{proof} Suppose that $T$ is a lattice triangle with it centroid $G$ on the lattice and that one of the sides of $T$ has lattice length $3n$, for some $n\in\N$.
By applying a map of the form $v\mapsto Mv+u$, where $M\in GL(2,\Z)$ and $u\in\Z^2$, we may assume that $T$ has vertices $O,A(3n,0),B(x,y)$, where $x,y\in\Z$ with $y>0$; cf.~the $x$-axis Lemma of \cite{Rab}. Since $G=\frac13(x+3n,y)$ is a lattice point, $x$ and $y$ must be divisible by $3$. Hence, $\gcd(x,y)$ and $\gcd(x-3n,y)$ are both multiples of $3$. So the lattice length of each of the sides of $T$ is a multiple of $3$.
\end{proof}


\section{Proof of Theorem \ref{T:3457}.}\label{S:3457}

First suppose $\ell=6 $ or $\ell\ge 8$. We will exhibit an acute lattice triangle of lattice perimeter $\ell$ for which the orthocenter is a lattice point.
If $\ell$ is even, say $\ell=2n$, with $n\ge 3$, consider the triangle $O,A=(n,0),B=(1,n-1)$. It has lattice perimeter $\ell$ and its orthocenter is the lattice point $(1,1)$. The angles $BOA$ and $OAB$ are acute, as $n>1$. The angle $ABO$ is acute since, calculating the vector dot product, $B\cdot (B-A)=(n-1)(n-2)$, which is positive for  $n\ge 3$.

If $\ell$ is odd with $\ell\ge 9$, we proceed as follows. If $\ell\equiv 1 \pmod 4$, 
 consider the triangle $O,A=(\frac12(\ell+1),0),B=(2,\frac12(\ell-3))$. It is easy to verify that $OAB$ has lattice perimeter $\ell$ (this is where we use the $\ell\equiv 1 \pmod 4$ assumption) and its orthocenter is the lattice point $(2,2)$. The angles $BOA$ and $OAB$ are acute, as $\ell>3$. The angle $ABO$ is acute since $B\cdot (B-A)=\frac14(\ell-3)(\ell-7)$, which is positive for $\ell\ge 9$.

If $\ell\equiv 3 \pmod 4$ with $\ell\ge 11$, consider the triangle $O, A=(\frac12(\ell+3),0), B=(6,\frac32(\ell-9))$. It has lattice perimeter $\ell$ and its orthocenter is the lattice point $(6,2)$. The angles $BOA$ and $OAB$ are acute, as $\ell>9$. The angle $ABO$ is acute since $B\cdot (B-A)=\frac34(\ell-9)(3\ell-31)$, which is positive for $\ell\ge 11$.

Now suppose $T$ is an acute lattice triangle  of lattice perimeter $\ell\in\{3,4,5,7\}$. We will derive a contradiction. Let the sides have lattice length $\ell_0,\ell_1,\ell_2$ (so $\ell_0+\ell_1+\ell_2=\ell$). 
Here is the complete list of the possible side lattice lengths $\ell_0,\ell_1,\ell_2$, each in increasing order; there are $8$ potential cases, which we label from $1$ to $8$:
\begin{align*}
\ell=3 &:\ (1)\ 1,1,1,\\
\ell=4 &:\ (2)\ 1,1,2,\\
\ell=5 &:\ (3)\ 1,1,3,\qquad (5)\ 1,2,2,\\
\ell=7 &:\ (4)\ 1,1,5,\qquad (6)\ 1,2,4,\qquad (7)\ 1,3,3,\qquad (8)\ 2,2,3.
\end{align*}
In cases $1-4$, there are two edges of side length $1$. So the orthocenter is not a lattice point in these cases, by Lemma~\ref{L:11n}.
In cases $5-8$, the three lattice lengths have no common prime factor, but there is a pair of lattice lengths that are not relatively prime. 
Hence by Lemma~\ref{L:gcd}, there is no lattice triangle with sides of these lattice lengths. This completes the proof of Theorem \ref{T:3457}.
\hfill $\qed$


\section{Proof of Theorem \ref{T:4610}.}\label{S:4610}

By Corollary~\ref{C:cor}, we need only consider triangles with even lattice perimeter. First suppose $\ell$ is even with $\ell=8 $ or $\ell\ge 12$. We will exhibit an acute lattice triangle of lattice perimeter $\ell$ for which the orthocenter is a lattice point.
Let $n\ge 1$, and set $A=(4n+2,0), B=(4,8n-4)$.  Computing the dot products $A\cdot B, A\cdot(A-B), B\cdot(B-A)$, one finds that the triangle $OAB$ is acute. The circumcenter is $F=(2n+1, 4n-3)$, since $|AF|^2=|BF|^2=|OF|^2=20n^2-20n+10$. The lattice perimeter of $OAB$ is $\ell=8n+4$. So this gives acute examples for $\ell \equiv 4 \pmod 8$ and $\ell \ge 12$. 

To get examples for $\ell$ of the form $8n+ 6$, consider $A=(2n+1,1),B=(0,6n+4)$, for $n\ge 2$. Here $F=(n-1,3n+2)$; the squared distances to the vertices are $5 + 10 n + 10 n^2$. Moreover, the angles are acute. So this gives examples for $\ell= 8n+ 6$ and $\ell\ge 22$.

For more of the required examples, let $n\ge 4$, and set $A=(2n+2,0), B=(4,2n-2)$. The triangle $OAB$ is acute and the circumcenter is $F=(n+1, n-3)$; the squared distances to the vertices are $2n^2-4n+10$. For $n$ even, $OAB$ has lattice perimeter $\ell=4n+2$, while for $n$ odd,  $\ell=4n+4$. So this gives acute examples for $\ell \equiv 0$ and $2 \pmod 8$ and $\ell \ge 18$. 

The above  examples miss four possible lattice perimeters, namely $\ell=8,12,14,16$. Here are examples in these cases:
\begin{itemize}
\item $\ell=8$;  $A(4,0),B(3,3)$. Here $F=(2,1)$.
\item $\ell=12$;  $A(6,0),B(1,5)$. Here $F=(3,2)$.
\item  $\ell=14$; $A(8,0),B(3,5)$. Here $F=(4,1)$.
\item  $\ell=16$; $A(8,0),B(1,7)$. Here $F=(4,3)$.
\end{itemize}

We now show that there are no such triangles in the cases $\ell=4,6,10$.  
The case $\ell=4$ is readily dispatched, since by Lemma~\ref{L:angle}, if the circumcenter  is a lattice point, the orthocenter would be a lattice point, but that is impossible for an acute lattice triangle when the lattice perimeter is $4$, by Theorem~\ref{T:3457}. It remains to treat the cases $\ell=6,10$.

Suppose that $T$ is an acute lattice triangle of lattice perimeter $\ell$ and that the circumcenter $F$ of $T$ is a lattice point. Let $\ell_0,\ell_1,\ell_2$ be the lattice lengths of the sides of the triangle, in increasing order, and let $\ell=\ell_0+\ell_1+\ell_2$. Let the angle opposite the side of lattice length $\ell_i$ be $\theta_i$, for $i=0,1,2$, respectively. 
First suppose that  $\ell=6$. Then there are three possibilities for $(\ell_0,\ell_1,\ell_2)$, namely $(1,1,4), (1,2,3), (2,2,2)$, but none of these satisfy the conclusions of Lemma~\ref{L:mid3}.
Finally,  suppose $\ell=10$. There are now 8 possibilities for $(\ell_0,\ell_1,\ell_2)$, namely 
\[
(1,1,8),\  (1,2,7),\  (1,3,6),\  (1,4,5),\  (2,2,6),\  (2,3,5),\  (2,4,4),\  (3,3,4).
\]
Of these, only five satisfy Lemma~\ref{L:mid3}, namely, 
\[
(1,3,6),\quad (1,4,5),\quad (2,3,5),\quad (2,4,4),\quad (3,3,4),
\]
 and of these only two satisfy Lemma~\ref{L:gcd}, namely $(1,4,5)$, and $(2,3,5)$. For the first case, $(\ell_0,\ell_1,\ell_2)=(1,4,5)$, Lemmas~\ref{L:theta} and \ref{L:angle} give
\[
\theta_0+\theta_1+\theta_2= \arctan(\frac1{m_0})+\arctan(\frac4{2m_1})+\arctan(\frac5{m_2}),
\]
for some positive integers $m_0,m_1,m_2$. Let us set 
\[
f(m_0,m_1,m_2)=\arctan(\frac1{m_0})+\arctan(\frac2{m_1})+\arctan(\frac5{m_2}),
\]
which is the RHS of the preceding expression. We claim that there are no values of $m_0,m_1,m_2$ for which $f(m_0,m_1,m_2)=\pi$. Indeed, calculations show that $f(m_0,m_1,m_2)>\pi$ for  $(m_0,m_1,m_2)=(1,1,1)$, and $f(m_0,m_1,m_2)<\pi$ for  $(m_0,m_1,m_2)=(1,1,2), (1,2,1)$ and $(2,1,1)$; see Table~\ref{T:145}. Hence $f(m_0,m_1,m_2)<\pi$ for all  $(m_0,m_1,m_2)$ with $m_0+m_1+m_2\ge 4$. 

\begin{table}
\begin{tabular}{c|c}
  \hline
  $m_0,m_1,m_2$ & $\frac1{\pi}(\arctan(\frac1{m_0})+\arctan(\frac2{m_1})+\arctan(\frac5{m_2}))$ (to 6 decimal places) \\ \hline
$1,1,1$ & $1.03958$\\
$1,1,2$ & $  0.981297$\\
$1,2,1$ & $  0.937167$\\
$2,1,1$ & $  0.937167$\\
\hline
\end{tabular}
\caption{Values in the $(\ell_0,\ell_1,\ell_2)=(1,4,5)$ case}\label{T:145}
\end{table}

For the remaining case, $(\ell_0,\ell_1,\ell_2)=(2,3,5)$, Lemmas~\ref{L:theta} and \ref{L:angle} give
\[
\theta_1+\theta_2+\theta_3= \arctan(\frac2{2m_0})+\arctan(\frac3{m_1})+\arctan(\frac5{m_2}),
\]
for some positive integers $m_0,m_1,m_2$. Let us define $g(m_0,m_1,m_2)=\arctan(\frac1{m_0})+\arctan(\frac3{m_1})+\arctan(\frac5{m_2})$, which is the RHS of the above expression. We claim that there is exactly one choice of values of $m_0,m_1,m_2$ for which $g(m_0,m_1,m_2)=\pi$, namely $(m_0,m_1,m_2)=(1,2,1)$. Indeed, calculations show that $g(m_0,m_1,m_2)>\pi$ for  $(m_0,m_1,m_2)=(1,1,1)$ and $(1,1,2)$, and that $g(m_0,m_1,m_2)<\pi$ for all values of $(m_0,m_1,m_2)$, other than $(1,1,1), (1,1,2)$ and $(1,2,1)$, with $m_0+m_1+m_2\le 5$; see Table~\ref{T:235}. It follows that $g(m_0,m_1,m_2)<\pi$ for all  $(m_0,m_1,m_2)$ with $m_0+m_1+m_2\ge 5$. Consequently, we have just one case left to eliminate: $(m_0,m_1,m_2)=(1,2,1)$.

\begin{table}
\begin{tabular}{c|c}
  \hline
  $m_0,m_1,m_2$ & $\frac1{\pi}(\arctan(\frac1{m_0})+\arctan(\frac3{m_1})+\arctan(\frac5{m_2}))$ (to 6 decimal places) \\ \hline
$1,1,1$ & $1.08475$\\
$1,1,2$ & $1.02646$\\
$1,2,1$ & $ 1$\\
$2,1,1$ & $  0.982334$\\
$1,1,3$ & $  0.975563$\\
$1,2,2$ & $ 0.941714$\\
$1,3,1$ & $  0.937167$\\
$2,1,2$ & $  0.924048$\\
$2,2,1$ & $  0.897584$\\
$3,1,1$ & $  0.937167$\\
\hline
\end{tabular}
\caption{Values in the $(\ell_0,\ell_1,\ell_2)=(2,3,5)$ case}\label{T:235}
\end{table}

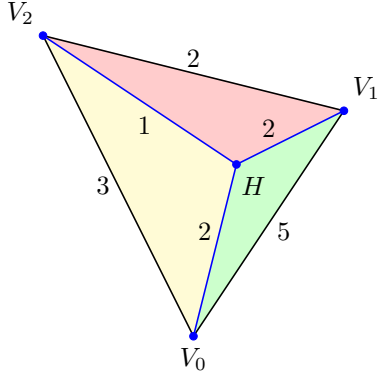
\begin{figure}
\begin{center}
\begin{tikzpicture}[scale=1]

\draw [fill, red!20] (4/7,16/7) 
  -- (2,3)  
  -- (-2,4)  
  -- cycle;

\draw [fill, green!20] (0,0) 
  -- (2,3)  
  -- (4/7,16/7)  
  -- cycle;
\draw [fill, yellow!20] (0,0) 
  -- (-2,4)  
  -- (4/7,16/7)  
  -- cycle;

\draw [black,semithick] (0,0) 
  -- (2,3)  
  -- (-2,4)   
 -- cycle;

\draw[-, blue,semithick] (0,0) -- (4/7,16/7);
\draw[-, blue,semithick] (2,3) -- (4/7,16/7);
\draw[-, blue,semithick] (-2,4) -- (4/7,16/7);

\draw[fill, blue] (0,0) circle (.05);
\draw[fill, blue] (2,3) circle (.05);
\draw[fill, blue] (-2,4) circle (.05);
\draw[fill, blue] (4/7,16/7) circle (.05);

\draw [black]  (0,-.3) node {$V_0$} ;
\draw [black]  (2.3,3.3) node {$V_1$} ;
\draw [black]  (-2.3,4.3) node {$V_2$} ;
\draw [black]  (.8,2) node {$H$} ;
\draw [black]  (0,3.7) node {$2$} ;
\draw [black]  (-1.2,2) node {$3$} ;
\draw [black]  (1.2,1.4) node {$5$} ;
\draw [black]  (.15,1.4) node {$2$} ;
\draw [black]  (1.,2.76) node {$2$} ;
\draw [black]  (-.65,2.8) node {$1$} ;
  \end{tikzpicture}
\caption{Triangle with $(\ell_0,\ell_1,\ell_2)=(2,3,5)$ and $(m_0,m_1,m_2)=(1,2,1)$.}\label{F:235121} 
\end{center}
\end{figure}

Let the vertices of the triangle be denoted $V_0,V_1,V_2$, such that $\ell_i$ is the lattice length of the side $V_{i+1}V_{i+2}$ opposite $V_i$, where the subscripts are computed modulo $3$. Recall the meaning of the numbers $m_0,m_1,m_2$ from Lemma~\ref{L:theta}:
\begin{itemize}
\item $2m_0$ is the lattice length of the segment from $H$ to $V_0$.
\item $m_1$ is the lattice length of the segment from $H$ to $V_1$.
\item $m_2$ is the lattice length of the segment from $H$ to $V_2$.
\end{itemize}
So, for $(\ell_0,\ell_1,\ell_2)=(2,3,5)$ and $(m_0,m_1,m_2)=(1,2,1)$, we have the general situation shown in Figure~\ref{F:235121}.
However, notice that the triangles $HV_0V_1$ and $HV_1V_2$ are lattice triangles and both violate Lemma~\ref{L:gcd}. So the case $(\ell_0,\ell_1,\ell_2)=(2,3,5)$ and $(m_0,m_1,m_2)=(1,2,1)$ is impossible. This completes the proof of Theorem \ref{T:4610}. \hfill $\qed$


\section{Proof of Theorem \ref{T:511}.}\label{S:511}

Suppose that $T$ is a lattice triangle of  lattice perimeter $\ell$ with its centroid $G$ on the lattice.
If $\ell=5$, then $T$ must have sides of lattice length $(1,1,3)$ or $(1,2,2)$, in increasing order, but the first case is impossible by Lemma~\ref{L:centroid3s}, and the second case is impossible by Lemma~\ref{L:gcd}.
Similarly, if $\ell=11$, then the lattice lengths of the sides of $T$ is one of the following $10$ cases, each in increasing order.
\begin{align*}
&(3,4,4),\quad (1,5,5),\quad (2,4,5),\quad (3,3,5),\quad (1,4,6),\\
&(2,3,6),\quad (1,3,7),\quad (2,2,7),\quad (1,2,8),\quad (1,1,9).
\end{align*}
Of these, $8$ are impossible by Lemma~\ref{L:gcd}, while the remaining cases, $(1,3,7)$ and $(1,1,9)$, are impossible by Lemma~\ref{L:centroid3s}. This proves the forward direction of the theorem.

Now suppose that $\ell\ge 3$ with $\ell\not\in\{5,11\}$. We will construct an acute lattice triangle $T$ of lattice perimeter $\ell$, whose centroid $G$ is  a lattice point.
First note that a triangle exists when $\ell=3$, namely $T=O,(1,2),(2,1)$, which has centroid $(1,1)$. 
For the rest of the proof  we will construct solutions of the form $T=O,A(z,0),B(x,y)$. In order to have $G\in\Z^2$, we require that $x+z$ and $y$ both be multiples of $3$.
In order to ensure that our triangle is acute, it suffices to choose $x$ such that $0<x<z$, and that $y$ is sufficiently large that the angle $OBA$ is acute.

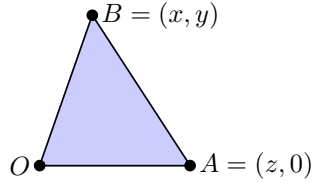
\begin{figure}
\begin{center}
\begin{tikzpicture}[>=latex]
  \draw[semithick] (0,0) coordinate(A) -- (0.7,2.0) coordinate(B) -- (2.,0) coordinate(C) -- cycle; 
  \draw[semithick,fill,blue!20] (0,0) coordinate(A) -- (0.7,2.0) coordinate(B) -- (2.,0) coordinate(C) -- cycle; 
  \draw[semithick] (0,0) coordinate(A) -- (0.7,2.0) coordinate(B) -- (2.,0) coordinate(C) -- cycle; 
\filldraw[black] (0,0) circle (2pt) node[anchor=east]{$O$}; 
 \filldraw[black] (0.7,2.0)  circle (2pt) node[anchor=west]{$B=(x,y)$};
\filldraw[black]  (2.,0) circle (2pt) node[anchor=west]{$A=(z,0)$};
\end{tikzpicture}
\caption{Acute triangle model.}
\end{center}
\end{figure}

It will be useful to first show that such triangles exist when $\ell$ is congruent to  $1$ or  $4$, modulo $6$. 

\begin{enumerate}[(i)]
\item[(i)] If $\ell\equiv 1\pmod 6$, set $z=\ell-2, x=4,y=(\ell-2)3^k$, for sufficiently large $k$. Note that $OA$ has lattice length $\ell-2$, $OB$ has lattice length $\gcd(4,(\ell-2)3^k)=1$ 
 and $AB$ has lattice length $\gcd(\ell-6,(\ell-2)3^k)=1$. 
Thus  $T=OAB$ has lattice perimeter $\ell$, independent of the choice of $k$. Note that $\ell\ge7$ and hence $x<z$. And $x+z$ and $y$ are both multiples of $3$, as required.

\item[(ii)] If $\ell\equiv 4\pmod 6$, set $z=\ell/2, x=1,y=\frac12(\ell-2)3^k$, for sufficiently large $k$. Then $OA$ has lattice length $\ell/2$, $OB$ has lattice length $\gcd(1,\frac12(\ell-2)3^k)=1$, and $AB$ has lattice length $\gcd(\frac12(\ell-2),\frac12(\ell-2)3^k=\frac12(\ell-2)$ 
. Thus  $T=OAB$ has lattice perimeter $\ell$. Note that $\ell\ge4$ and hence $x<z$. And $x+z$ and $y$ are both multiples of $3$, as required.
\end{enumerate}

For convenience, let $\mathcal A_G(\ell)$ denote the set of  integers $\ell$ for which there exists an acute lattice triangle $T$ of lattice $\ell$, whose centroid is a lattice point.
From above, there is no triangle $T\in \mathcal A_G(\ell)$ for $\ell=2,5$  or $11$. But if  we have a triangle $T\in \mathcal A_G(\ell)$ for some $\ell$, then multiplying $T$ by an integer $n$, we obtain a triangle in $\mathcal A_G(n\ell)$. So to complete the proof, it suffices to show that $\mathcal A_G(\ell)$ is non-empty for $\ell\in\{2^2,5^2,11^2,2\cdot 5,2\cdot 11,5\cdot 11\}$, and for each prime $\ell$ other than $2, 5$ and $11$. Note that $5^2,11^2,5\cdot 11$ are congruent to $1$ modulo $6$, while $2^2,2\cdot 5,2\cdot 11$ are congruent to $4$ modulo $6$. So these $6$ integers are covered by constructions (i) and (ii) above. Consequently, using (i) and (ii)  again, it remains to treat the case where $\ell$ is prime,  $\ell\equiv 5\pmod 6$, and $\ell\not\in\{5,11\}$.

We further examine the possibilities according the value of $\ell$ modulo $18$. As $\ell\equiv 5\pmod 6$, there are $3$ cases: $\ell\equiv 5, 11$ or $17$, modulo $18$. 
It is convenient to introduce some notation. If  $n\in \N$, let $\delta(n)$ denote the smallest prime divisor of $n$ that is congruent to $5$ modulo $6$, if it exists, and set $\delta(n)=\infty$ otherwise. As before, we will construct triangles of the form  $T=O, A(z,0), B(x,y)$ with $0<x<z$, and $x+z,y$ are multiples of $3$, and  $y$ is sufficiently large that the angle $OBA$ is acute.

\begin{enumerate}
\item[(iii)] Suppose $\ell\equiv 5\pmod{18}$ with $\ell\not=5$. Then $\ell-8\equiv 15\pmod{18}$, and hence $\frac13(\ell-8)\equiv 5\pmod{6}$ and so $\frac13(\ell-8)$ has  a prime divisor congruent to $5$ modulo $6$. Hence  $\delta(\ell-8)$ is well defined. Now set $z=\ell-1-\delta(\ell-8), x=7,y=\delta(\ell-8)\cdot 3^k$, for sufficiently large $k$. Note that as $\ell$ and $\delta(\ell-8)$ are congruent to $5$ modulo $6$, we have $x+z\equiv 0\pmod3$, as required. Moreover, $\delta(\ell-8)<\ell-8$, so $z>7=x$.
Also, $OA$ has lattice length $\ell-1-\delta(\ell-8)$, $OB$ has lattice length $\gcd(7,\delta(\ell-8)\cdot 3^k)=1$, since $3, 7, \delta(\ell-8)$ are three distinct primes, and $AB$ has lattice length $\gcd(\ell-8-\delta(\ell-8),\delta(\ell-8)\cdot 3^k) = \delta(\ell-8)\gcd(\frac{\ell-8}{\delta(\ell-8)} -1,3^k)=\delta(\ell-8)$. Thus  $T=OAB$ has lattice perimeter $\ell$, as required.

\item[(iv)] Suppose $\ell\equiv 11\pmod{18}$ with $\ell\not=11$. Then $\ell-14\equiv 15\pmod{18}$ and, arguing as in case (iii), $\frac13(\ell-14)\equiv 5\pmod{6}$ and it follows that  $\frac{1}{3}(\ell-14)$ has a prime divisor congruent to $5$ modulo $6$.
Now set $z=\ell-1-\delta(\ell-14), x=13,y=\delta(\ell-14)\cdot 3^k$, for sufficiently large $k$. 
Then,  as $\ell$ and $\delta(\ell-14)$ are congruent to $11$ modulo $18$, we have $x+z\equiv 0\pmod3$. We also have $\delta(\ell-14)<\ell-14$, so $z>13=x$.
Furthermore,
$OA$ has lattice length $\ell-1-\delta(\ell-14)$, $OB$ has lattice length $\gcd(13,\delta(\ell-14)\cdot 3^k)=1$, since $13, \delta(\ell-14), 3 $ are three distinct primes, and $AB$ has lattice length $\gcd(\ell-14-\delta(\ell-14),\delta(\ell-14)\cdot 3^k)= \delta(\ell-14)\gcd(\frac{\ell-14}{\delta(\ell-14)} - 1, 3^k)=\delta(\ell-14)$. Thus  $T=OAB$ has lattice perimeter $\ell$, as required.

\item[(v)] Suppose $\ell\equiv 17\pmod{18}$. Then $\ell-2\equiv 15\pmod{18}$ and, arguing again as in case (iii), $\frac13(\ell-2)\equiv 5\pmod{6}$ and it follows that  $\frac{1}{3}\delta(\ell-2)$ has a primes divisor congruent to $5$ modulo $6$.
Now set $z=\ell-1-\delta(\ell-2), x=1,y=\delta(\ell-2)\cdot 3^k$, for sufficiently large $k$. 
Then, as in the two previous cases, $x+z\equiv 0\pmod3$ and $z>x$. Finally, $OA$ has lattice length $\ell-1-\delta(\ell-2)$, $OB$ has lattice length $\gcd(1,\delta(\ell-2)\cdot 3^k)=1$, and $AB$ has lattice length $\gcd(\ell-2-\delta(\ell-2),\delta(\ell-2)\cdot 3^k)= \delta(\ell-2)\gcd(\frac{\ell-2}{\delta(\ell-2)} - 1, 3^k)=\delta(\ell-2)$. Thus  $T=OAB$ has lattice perimeter $\ell$, as required.
\end{enumerate}

This completes the proof of Theorem~\ref{T:511}.\hfill $\qed$


\section{Obtuse and right triangles, and combining hypotheses.}\label{S:obright}

Let us first consider the orthocenter.  For all $\ell\ge 3$, the obtuse lattice triangle with vertices $O,(1,0),(2-\ell,\ell-2)$ has orthocenter $(2-\ell,1-\ell)$ and lattice perimeter $\ell$. 
All right lattice triangles have their orthocenter on the integer lattice, at their right angle vertex, and there are right lattice triangles of arbitrary lattice perimeter $\ell\ge 3$; for example, one can take $O,(\ell-2,0),(0,1)$. Thus, for obtuse and right lattice triangles, the hypothesis that the orthocenter is a lattice point has no consequence for the lattice perimeter.

Now consider the circumcenter.
For all even $\ell\ge 4$,  the obtuse triangle $O,A=(2,0), B=(3-\ell,\ell-3)$ has circumcenter $(1,\ell-2)$ and lattice perimeter $\ell$. 
Similarly, for all even $\ell\ge 4$,  the right triangle $O,A=(1,1),B= (3-\ell,\ell-3)$ has circumcenter   $\frac12(4-\ell,\ell-2)$  and lattice perimeter  $\ell$.
Thus, for obtuse and right lattice triangles, the hypothesis that the circumcenter is a lattice point has no consequence for the lattice perimeter other than the requirement, imposed for all lattice triangles by Corollary ~\ref{C:cor}, that the lattice perimeter be even.

Now the centroid. First note that our proof of Theorem~\ref{T:511}, in Section~\ref{S:511}, our proof that $\ell\not\in\{5,11\}$ did not use the hypothesis that the triangle is acute! So this restriction holds for all lattice triangles whose centroid is a lattice point. Indeed, we claim that Theorem~\ref{T:511} remains valid with the word ``acute'' replaced by ``obtuse''. 
To see this we just need to exhibit the triangles in question. For the lattice perimeter $3$ case, note that the obtuse triangle $O,A(1,0),B(-1,3)$ has lattice perimeter $3$ and  centroid $(0,1)$. 
For $\ell\ge 4$ with $\ell\not\in\{5,11\}$, we constructed  an acute lattice triangle $T$ of the form $T=O,A(z,0),B(x,y)$, of  lattice perimeter $\ell$ with its centroid $G$ on the lattice. By applying a shear transformation using a matrix of the form $\begin{psmallmatrix}1&-k\\0&1\end{psmallmatrix}$, for sufficiently large $k\in\N$, one obtains an obtuse lattice triangle of  lattice perimeter $\ell$. Since linear maps preserve the property that the centroid is a lattice point, this gives an obtuse lattice triangle of lattice perimeter $\ell$ with its centroid on the lattice. 
Thus, for obtuse  lattice triangles, the hypothesis that the circumcenter is a lattice point has the same restriction on the lattice perimeter as for acute lattice triangles.

For right triangles whose centroid is a lattice point, the situation is a little different.
Suppose that $T$ is a right triangle lattice whose centroid is a lattice point. By translation, and reflection in the axes, we may assume that the right angle is at the origin. That is, we have $T=O,A(x,y),B(z,w)$, and since $AOB$ is a right angle, there exists $(u,v)\in\Z^2$, with $\gcd(u,v)=1$, and integers $i,j$ such that $(x,y)=i(u,v)$ and $(z,w)=j(-v,u)$. As the centroid $G=\frac13((x,y)+(z,w))$ of $T$ is on the integer lattice, we have $iu-jv\pmod3$ and $iv+ju\pmod3$. That is, working modular $3$,
\[
\begin{pmatrix}
i&-j\\
j&i
\end{pmatrix}\begin{pmatrix}u\\ v\end{pmatrix}\equiv \begin{pmatrix}0\\0\end{pmatrix}.
\]
Let $\Delta= \begin{psmallmatrix}
i&-j\\
j&i
\end{psmallmatrix}$. If $\det\Delta \not\equiv0\pmod3$, then $\begin{psmallmatrix}u\\v\end{psmallmatrix}\equiv \Delta^{-1}\begin{psmallmatrix}0\\0\end{psmallmatrix}\equiv \begin{psmallmatrix}0\\0\end{psmallmatrix}$, in which case $u,v$ are both multiples of $3$. If $\det\Delta \equiv0\pmod3$, then $i^2+j^2\equiv0\pmod3$, in which case $i,j$ are necessarily both multiples of $3$. So in either case, $x,y,z,w$ are all multiples of $3$. It follows that the lattice length of the sides of $T$ are all multiples of $3$. In particular, the smallest possible lattice perimeter is $9$. 
Note that the right triangle with vertices $O,(3n,0),(0,3)$ has centroid $(n,1)$ and lattice perimeter $3n+6$. So any multiple of $3$ greater than $6$ can be obtained. 

Finally, we discuss  what happens when we assume that more than one of the triangle centers $F,G,H$ lie on the lattice.
When the circumcenter $F$ is a lattice point, the orthocenter $H$ is also necessarily a lattice point, by Lemma~\ref{L:angle}. So the condition ``$F$ and $H$ are lattice points'' is the same as the condition ``$F$ is a lattice point'', and the condition ``$F,G,H$ are lattice points'' is the same as the condition ``$F$ and $G$ are lattice points''. Thus, the only combinations that need be considered are: (a) $G$ and $H$ are lattice points, and (b)  $F,G,H$ are lattice points.

Case (a). We claim that if $\ell\in\N$, then there exists an acute lattice triangle of lattice perimeter $\ell$, for which the centroid $G$ and orthocenter $H$ are both lattice points, if and only if $\ell$ is a multiple of $3$ with $\ell\ge 9$. Moreover, we claim that this statement remains true if the word ``acute'' is replaced by ``obtuse'', or by ``right''. We start with a general observation.

\begin{lemma}\label{L:orthocentroid3s} Suppose that $T$ is a lattice triangle with its centroid $G$ and orthocenter $H$ on the lattice. 
Then  the lattice length of each side  of $T$ is  a multiple of $3$. In particular, the lattice perimeter $\ell$ of $T$ is  a multiple of $3$ and $\ell\ge 9$.
\end{lemma}

\begin{proof} 
Suppose, without loss of generality,  that $T$ has vertices $0,A=(x_1,y_1),B=(x_2,y_2)$. The orthocenter $H$ of $T$ is
\[
H=\frac{x_1x_2+y_1y_2}{x_1y_2-x_2y_1}( y_2-y_1, x_1-x_2).
\]
Since $G=\frac13(x_1+x_2,y_1+y_2)\in\Z^2$, we have $x_2\equiv -x_1\pmod 3$ and $y_2\equiv -y_1\pmod 3$.
Thus $x_1y_2-x_2y_1\equiv 0\pmod 3$. Thus, since $H\in\Z^2$, either $x_1x_2+y_1y_2$ is a multiple of $3$, or $y_2-y_1$ and $x_1-x_2$ are both multiples of $3$. It follows that, since $x_2\equiv -x_1\pmod 3$ and $y_2\equiv -y_1\pmod 3$, the numbers  $x_1,x_2,y_1,y_2$ are all multiples of $3$. Hence, 
 the lattice length of each side  of $T$ is  a multiple of $3$. 
\end{proof}

Note that if the orthocenter of a lattice triangle $T$ is a lattice point, then the centroid  and orthocenter of the triangle $3T$ are both lattice points. Hence, from the discussion in the first paragraph of this section, by multiplying by $3$, there exist obtuse and right lattice triangles, with $G,H\in\Z^2$, of every lattice perimeter $\ell$ that is a multiple of $3$ with $\ell\ge 9$.
Furthermore, by Theorem~\ref{T:3457}, there exist acute lattice triangles, with $H\in\Z^2$, of every lattice perimeter $\ell\not\in\{3,4,5,7\}$, and hence by multiplying by $3$ we obtain acute lattice triangles, with $H\in\Z^2$, of every lattice perimeter $\ell\in3\N$ with $\ell\ge 9$, except possibly for $\ell\in\{9,12,15,21\}$. For these remaining four cases, here are explicit acute triangles:
\begin{itemize}
\item {$\ell=9:$}  the triangle $O,(6,3),(3,6)$ has centroid  $(3,3)$ and orthocenter $(4,4)$.
\item {$\ell=12:$} the triangle $O,(6,0),(3,9)$ has centroid  $(3,3)$ and orthocenter $(3,1)$.
\item {$\ell=15:$}  the triangle $O,(9,0),(3,9)$ has centroid  $(4,3)$ and orthocenter $(3,2)$.
\item {$\ell=21:$}  the  triangle $O,(15,0),(3,9)$ has centroid  $(6,3)$ and orthocenter $(3,4)$.
\end{itemize}

Case (b). We claim that if $\ell\in\N$, then there exists an acute lattice triangle of lattice perimeter $\ell$, for which $F,G,H\in\Z^2$, if and only if $\ell$ is a multiple of $6$ and $\ell\ge 12$. Moreover, we claim that this statement remains true if the word ``acute'' is replaced by ``obtuse'', or by ``right''. 

If a lattice triangle has lattice perimeter $\ell$, and $F,G,H\in\Z^2$, then $\ell$ is even by Corollary~\ref{C:cor}, and  $\ell$ is a multiple of $3$ and $\ell\ge 9$, by Lemma~\ref{L:orthocentroid3s}; hence $\ell$ is necessarily a multiple of $6$ and $\ell\ge 12$. It remains to see that such triangles exist in each of the three cases: obtuse, right and acute.
From the discussion in the second paragraph of this section, there exist obtuse and right lattice triangles, with $F\in\Z^2$, of all even lattice perimeters  $\ell$  with $\ell\ge 4$. 
It follows that, by multiplying by $3$, there exist obtuse and right lattice triangles, with $F,G,H\in\Z^2$, for all that is a multiple of $6$ with $\ell\ge 12$.
By Theorem \ref{T:4610}, there exist acute lattice triangles, with $F\in\Z^2$,  for all even lattice perimeters $\ell$ with $\ell\not\in\{4,6,10\}$. 
Multiplying by $3$, we obtain the desired acute triangles, except possibly for lattice perimeters $\ell\in\{12,18,30\}$.
For these remaining three cases, here are explicit acute triangles:
\begin{itemize}
\item {$\ell=12:$} the triangle $O,(6,0),(3,9)$ has  $F=(3,4),G=(3,3),H=(3,1)$.
\item {$\ell=18:$}  the triangle $O,(12,6),(6,12)$ has  $F=(5,5),G=(6,6),H=(8,8)$.
\item {$\ell=30:$} the triangle $O,(18,0),(6,18)$ has  $F=(9,7),G=(8,6),H=(6,4)$.
\end{itemize}

To conclude, the results of Theorems \ref{T:3457}, \ref{T:4610}, \ref{T:511} are summarized in Table~\ref{Tab:results}, which gives the values of the lattice perimeters $\ell$ that are achievable in the different cases.

\begin{table}
\begin{tabular}{c||l|l|l}
  \hline
  Points in $\Z^2$ & Acute&Obtuse&Right \\ \hline
$F$ &$2\N\backslash\{2,4,6,10\}$&$2\N\backslash\{2\}$&$2\N\backslash\{2\}$\\
 $G$&$\N\backslash\{1,2,5,11\}$&$\N\backslash\{1,2,5,11\}$&$3\N$\\
  $H$&$\N\backslash\{1,2,3,4,5,7\}$&$\N\backslash\{1,2\}$&$\N\backslash\{1,2\}$ \\
   $G,H$&$3\N\backslash\{3,6\}$&$3\N\backslash\{3,6\}$&$3\N\backslash\{3,6\}$\\
    $F,G,H$&$6\N\backslash\{6\}$&$6\N\backslash\{6\}$&$6\N\backslash\{6\}$\\
\hline
\end{tabular}
\vskip.1cm
\caption{Achievable lattice perimeters}\label{Tab:results}
\end{table}

\section{An open problem.}\label{S:open}

\begin{figure}
\begin{center}
\begin{tikzpicture}[>=latex]
  \draw[semithick,fill,blue!20] (0,0) coordinate(A) -- (4,0) coordinate(B) -- (4,3) coordinate(C) -- cycle; 
  \draw[semithick] (0,0) coordinate(A) -- (4,0) coordinate(B) -- (4,3) coordinate(C) -- cycle; 
\filldraw[black] (0,0) circle (2pt) node[anchor=north east]{$O$}; 
 \filldraw[black] (4,0)  circle (2pt) node[anchor=north west]{$A=(4,0)$};
\filldraw[black]  (4,3) circle (2pt) node[anchor=north west]{$B=(4,3)$};
\filldraw[black] (3,1) circle (2pt) node[anchor=north]{$I=(3,1)$}; 
\draw[blue,thick]  (3,1) circle (1);

\foreach \ii in {-1,...,3}
\draw[-,dotted,semithick] (-1,\ii) -- (5,\ii);
\foreach \jj in {-1,...,5}
\draw[-,dotted,semithick] (\jj,-1) -- (\jj,3);

\draw[->,semithick] (-1,0) -- (5,0);
\draw[->,semithick] (0,-1) -- (0,3);

\filldraw[red] (2.4,1.8) circle (2pt); 
\filldraw[red] (4,1) circle (2pt); 
\filldraw[red] (3,0) circle (2pt); 
\draw[black] (2.4,1.8) circle (2pt); 
\draw[black] (4,1) circle (2pt); 
\draw[black] (3,0) circle (2pt); 

\end{tikzpicture}
\caption{The incenter of the $3,4,5$ right triangle.}\label{F:inc}
\end{center}
\end{figure}
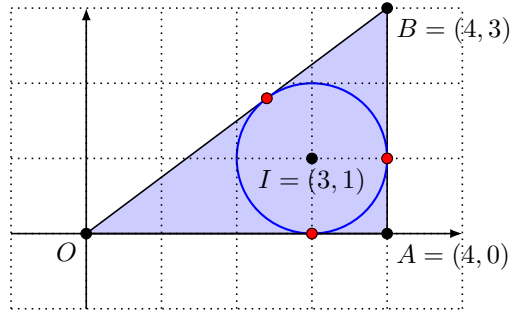

We have examined the consequences of the circumcenter, centroid and orthocenter being a lattice point. An obvious further case to study is that of the \emph{incenter}; see Figure~\ref{F:inc}.
For a lattice triangle $T$ with vertices $O,A=(x_1,y_1),B=(x_2,y_2)$ and sides $a=|A|,b=|B|,c=|A-B|$, the incenter of $T$ is given by the following formula.
\begin{equation}\label{equ:I}
I=\frac{aB+bA}{a+b+c}.
\end{equation}
A cursory examination of examples indicates that for the incenter there should be results analogous to those of Theorems~\ref{T:3457}, \ref{T:4610}, \ref{T:511}, but with different numbers. The authors of this paper have not undertaken a study of this problem, and as far as we know, it has not been treated in the literature. We believe it might make an interesting project. It entails a nice programing exercise to generate examples and thus formulate conjectures. Then there is the task of devising series of explicit examples showing that certain lattice perimeters are possible with the incenter on the lattice, and finally the more difficult task of developing proofs that exclude the remaining cases.
Not having undertaken this project ourselves, we give no guarantee that the difficulty is not greater than what we experienced in the proof of the results presented in this paper.
In fact, we suspect it might be more difficult. Here is an example of the sort of situation that can arise.

\begin{example} In Figure~\ref{F:inc}, the three points at which the incircle touches the triangle are not all lattice points; only two of them are. Consider the lattice triangle $T$ with vertices 
$O,(14,2),(8,8)$. Its incenter is the lattice point $(8,4)$, and not only are the three points at which the incircle touches $T$ not all lattice points, but the inradius of $T$ is irrational; it is $2\sqrt2$; see Figure~\ref{F:inc2}. Not surprisingly, there are larger lattice triangles which have their incenter on the lattice, an irrational inradius, and none of the points at which the incircle touches $T$ are lattice points. For example, consider the triangle with vertices $O,(14,2),(21,51)$.
 It has incenter $(9,7)$, inradius $4 \sqrt2$ and the points where the  incircle touches the triangles are  $\frac15(49,7),\frac15(73,31),\frac1{13}(49,119)$.

\begin{figure}
\begin{center}
\begin{tikzpicture}[>=latex]
  \draw[semithick,fill,blue!20] (0,0) coordinate(A) -- (4,4) coordinate(B) -- (7,1) coordinate(C) -- cycle; 
  \draw[semithick] (0,0) coordinate(A) -- (4,4) coordinate(B) -- (7,1) coordinate(C) -- cycle; 
\filldraw[black] (0,0) circle (2pt) node[anchor=north east]{$O$}; 
 \filldraw[black] (4,4)  circle (2pt) node[anchor=west]{$(8,8)$};
\filldraw[black]  (7,1) circle (2pt) node[anchor=west]{$(14,2)$};
\filldraw[black] (4,2) circle (2pt) node[anchor=north]{$(8,4)$}; 
\draw[blue,thick]  (4,2) circle (1.416);

\foreach \ii in {-1,...,8}
\draw[-,dotted,semithick] (-1/2,\ii/2) -- (7,\ii/2);
\foreach \jj in {-1,...,14}
\draw[-,dotted,semithick] (\jj/2,-1/2) -- (\jj/2,4);

\draw[->,semithick] (-1/2,0) -- (7,0);
\draw[->,semithick] (0,-1/2) -- (0,4);

\filldraw[red] (4.2,.6) circle (2pt); 
\filldraw[red] (3,3) circle (2pt); 
\filldraw[red] (5,3) circle (2pt); 
\draw[black] (4.2,.6) circle (2pt); 
\draw[black] (3,3) circle (2pt); 
\draw[black] (5,3) circle (2pt); 

\end{tikzpicture}
\caption{A lattice triangle with its incenter on $\Z^2$ but irrational inradius.}\label{F:inc2}
\end{center}
\end{figure}
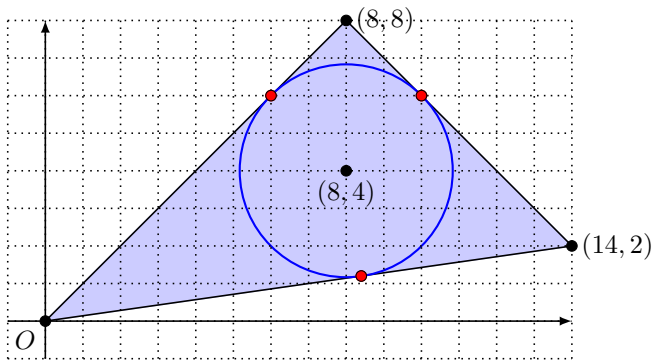

\end{example}

\bigskip
\noindent
{\bf Author contributions statement.} 
The two authors were both authors in the traditional sense of authorship of a mathematics paper. Both authors agree to be accountable for all aspects of the work.

\bigskip
\noindent
{\bf Competing interests statement.} 
There are no interests to declare.

\bigskip
\noindent
{\bf Declaration of funding statement.} 
No funding was received.



\begin{thebibliography}{1}

\bibitem{Rab}
Rabinowitz S. 
A census of convex lattice polygons with at most one interior lattice point.
Ars Combin. 1989;28:{83--96}.

\bibitem{HSk}
{Hille L, Skarke H.}
{Reflexive polytopes in dimension 2 and certain relations in
              {${\rm SL}_2(\mathbb Z)$}}.
{J. Algebra Appl.}
 2002;1:{159--173}.


\bibitem{WeiDing}
{Wei X, Ding R.}
{Lattice polygons with two interior lattice points},
{Translation of Mat. Zametki {{\bf{9}}1} (2012), no. 6,
              920--933}.
{Math. Notes}
  2012;91:{868--877}.
  
\bibitem{Worm}
{Wormleighton B.}
{Algebraic capacities}.
{Selecta Math. (N.S.)}
2022;28:{Paper No. 9, 45}.

\bibitem{HaaseNil}
{Haase C, Nill B, Paffenholz A.}
{Lecture Notes on Lattice Polytopes}.
{\url{https://www2.mathematik.tu-darmstadt.de/~paffenholz/daten/preprints/20210628_Lattice_Polytopes.pdf}}.
{Technische Universit{\"a}t Berlin}; 2021.

\bibitem{Scott}
{Scott PR.}
{The fascination of the elementary}.
{Amer. Math. Monthly}
1987; {94}:{759--768}.

\bibitem{HaaSc}
{Haase C, Schicho J.}
{Lattice polygons and the number {$2i+7$}}.
{Amer. Math. Monthly}
 2009;116:{151--165}.

\bibitem{Abrams}
{Abrams A, Pommersheim J}.
{Integer area dissections of lattice polygons via a non-abelian {S}perner's {L}emma}.
{Amer. Math. Monthly}
2025;132:{737--753}.

\bibitem{PoRV}
{Poonen B, Rodriguez-Villegas F}.
{Lattice polygons and the number 12}.
{Amer. Math. Monthly}
 2000; {107}:{238--250}.
 
\bibitem{AN}
Alsina C, Nelsen RB.
Icons of mathematics, an exploration of twenty key images.  
Dolciani Mathematical Expositions vol.~45.
Washington DC: Mathematical Association of America; 2011.

\bibitem{AN2}
Alsina C, Nelsen RB.
Charming proofs, a journey into elegant mathematics. 
Dolciani Mathematical Expositions vol.~42.
Washington DC: Mathematical Association of America; 2010.





\end{thebibliography}
\end{document}